\documentclass[a4paper,10pt]{article}
\usepackage{graphicx, color}
\usepackage{amsmath,amssymb,amsthm,comment,cite}
\usepackage{enumitem}
\usepackage[colorlinks=true,linkcolor=blue,citecolor=blue,urlcolor=blue]{hyperref}

\newtheorem{cor}{Corollary}
\newtheorem{obs}{Observation}
\newtheorem{prop}{Proposition}

\newtheorem{lem}{Lemma}
\newtheorem{rem}{Remark}

\newtheorem{quest}{Question}
\newtheorem{conj}{Conjecture}

\theoremstyle{definition}
\newtheorem{defi}{Definition}[section]

\usepackage[noabbrev,capitalize,nameinlink]{cleveref}

\crefname{prop}{Proposition}{propositions}
\crefname{prob}{Problem}{problems}
\crefname{lem}{Lemma}{lemmas}
\crefname{rem}{Remark}{remarks}
\crefname{tab}{Table}{tables}
\crefname{quest}{Question}{questions}
\crefname{conj}{Conjecture}{conjectures}
\crefname{defi}{Definition}{definitions}
\crefname{claim}{Claim}{claims}
\Crefname{thm}{Theorem}{Theorems}
\Crefname{cor}{Corollary}{Corollaries}
 \Crefname{obs}{Observation}{Observations}

\usepackage{thm-restate}

\title{Merino--Welsh inequalities for matroids with controlled lattices of cyclic flats}
\author{Kolja Knauer\thanks{Departament de Matem\'{a}tiques i Inform\'{a}tica, Universitat de Barcelona, Centre de Recerca Matemàtica, Barcelona, Spain.} \and Leonardo Martínez-Sandoval\thanks{Facultad de Ciencias, Universidad Nacional Autónoma de México, Ciudad de México, México.} \and Criel Merino\thanks{Instituto de Matemáticas, Universidad Nacional Autónoma de México, Ciudad de México, México.}}
\date{}

\begin{document}

\maketitle

\begin{abstract}
Beke, Cs\'aji, Csikv\'ari, and Pituk showed that the Merino--Welsh quotient $\Phi(M)=T_M(1,1)^2/(T_M(2,0)T_M(0,2))$ can be arbitrarily large, so the multiplicative Merino--Welsh inequality fails for matroids in general. We show that $\Phi$ is uniformly bounded on the class of matroids whose cyclic-flat lattice avoids any fixed finite poset $P$ as an induced subposet. We further prove $\Phi(M)\leq1$ for matroids of cyclic width at most $7$, cyclic height at most $6$, and for loop- and coloop-free $4$-paving or $4$-copaving matroids.
\end{abstract}

\medskip
\noindent\textbf{2020 Mathematics Subject Classification.} Primary 05B35; Secondary 05C31, 06A07.

\smallskip
\noindent\textbf{Keywords.} Tutte polynomial, Merino--Welsh problem, matroid, cyclic flat, cyclic width, cyclic height, paving matroid.

\section{Introduction}

 For a graph $G$ or a matroid $M$, we denote its Tutte polynomial by $T_G$ or $T_M$, respectively. We assume familiarity with the theory of Tutte polynomials as presented in~\cite[Chapter 6]{white1992matroid}.
\subsection{State of the art}

Historically, Merino and Welsh explicitly posed the first of the following inequalities as Conjecture~7.1 in~\cite{MW99}. Conde and Merino later explicitly posed the multiplicative and additive strengthenings as Conjectures~2.2 and~2.3 in~\cite{CM09}.

\begin{conj}[Graphic Merino--Welsh inequalities]\label{prob:graphic-mw}
Let $G$ be a connected graph without loops or bridges. Then each of the following inequalities is true.
\begin{itemize}
\item[] \begin{equation}\label{eq:intro-mw-original}
\max\{T_G(2,0),T_G(0,2)\}\ge T_G(1,1).
\end{equation}
\item[] \begin{equation}\label{eq:intro-mw-additive}
T_G(2,0)+T_G(0,2)\ge 2T_G(1,1).
\end{equation}
\item[] \begin{equation}\label{eq:intro-mw-multiplicative}
T_G(2,0)T_G(0,2)\ge T_G(1,1)^2.
\end{equation}
\end{itemize}
\end{conj}
The multiplicative inequality implies the additive one, which in turn implies~\eqref{eq:intro-mw-original}. We refer to~\eqref{eq:intro-mw-multiplicative} as the \emph{multiplicative Merino--Welsh inequality}.


The same inequalities extend naturally from graphs to the Tutte polynomial of a loop- and coloop-free matroid $M$. The matroidal maximum inequality had already been conjectured in~\cite{MRR09} for matroids that contain two disjoint bases or whose ground set is the union of two bases; the three unrestricted matroidal versions were explicitly formulated as Conjecture~1.2 in~\cite{KMR18}.
These matroidal inequalities were subsequently verified on a number of structurally different classes. For lattice path matroids, Knauer, Mart\'inez-Sandoval and Ram\'irez Alfons\'in proved a strengthening of the multiplicative inequality: apart from direct sums of parallel pairs, the right-hand side can be multiplied by $4/3$~\cite{KMR18}. Ch\'avez-Lomel\'i, Merino, Noble, and Ram\'irez-Ib\'a\~nez proved a convexity statement for the Tutte polynomial of coloopless paving matroids, implying the additive inequality~\cite{CMNR11}. Ferroni and Schr\"oter later proved the multiplicative inequality for all split matroids, a class strictly containing the paving and copaving matroids~\cite{FS23}.

The situation changed decisively with the counterexamples of Beke, Cs\'aji, Csikv\'ari, and Pituk~\cite{BCCP24false}. They showed that all three matroidal versions can fail. Moreover, setting 

\begin{equation}\label{eq:intro-phi}
\Phi(M)=\frac{T_M(1,1)^2}{T_M(2,0)T_M(0,2)}.
\end{equation}
 they showed $\Phi(M)$ is not bounded in general. Their examples are obtained by replacing every element of suitable uniform matroids by a parallel pair. Thus, for matroids, the question is no longer whether the Merino--Welsh inequalities hold universally, but rather what remains true under natural structural assumptions.

After the counterexample, there are at least two ways to weaken the universal statement: one can move the evaluation points away from $(2,0)$ and $(0,2)$, or one can restrict the matroid class. The first direction goes back to Jackson's universal inequality $T_M(3,0)T_M(0,3)\ge T_M(1,1)^2$~\cite{Jac10}. Beke, Cs\'aji, Csikv\'ari, and Pituk lowered the constant $3$ to $2.9243$~\cite{BCCP24perm}, and Csikv\'ari recently lowered it further to $2.355$~\cite{Csi26}. In a complementary structural direction, Kung obtained sufficient density and cocircuit-size conditions for the additive inequality~\cite{Kun25}.

\subsection{Our contributions}

Our paper asks which structural restrictions on a matroid class prevent $\Phi(M)$ from becoming arbitrarily large. Hence, for a class $\mathcal M$ of matroids, we want to bound the parameter
\begin{equation}\label{eq:intro-PHI}
\Phi(\mathcal M)=\sup\{\Phi(M)\mid M\in\mathcal M\text{ is loop- and coloop-free}\}.
\end{equation}

The first main tool we develop  is an extremal reduction for minor-closed classes $\mathcal{M}$. Starting with any loop- and coloop-free matroid, deletion--contraction either produces a smaller minor with at least as large a value of~$\Phi$, or forces every element into a nontrivial parallel or coparallel class. We then show that classes of size at least three can also be reduced without decreasing $\Phi$. Thus, when maximizing $\Phi$ over a minor-closed class, it is enough to consider \emph{$2$-inflated matroids}: every element belongs to exactly one maximal nontrivial parallel or coparallel class, and every such class has size two. This is made precise in the following lemma.

\begin{restatable}{lem}{basislem}\label{lem:inductionbasis}
If $\mathcal M$ is a minor-closed class of matroids, then
\[
\Phi(\mathcal M)=\sup\{\Phi(M)\mid M\in\mathcal M\text{ is $2$-inflated}\}.
\]
\end{restatable}

In view of both the equality cases for lattice path matroids~\cite{KMR18} and the $2$-thickening used in the known counterexamples~\cite{BCCP24false}, this gives a structural explanation for the recurring role of parallel and coparallel pairs. In \cref{sec:core} we associate to a $2$-inflated matroid a smaller \emph{decorated core} and two corresponding \emph{boundary matroids}, and obtain bounds for its Merino--Welsh quotient in terms of these boundary matroids. 

The main focus of the present paper is on the Merino-Welsh inequalities in relation to the lattice of cyclic flats. One reason to use this object is that the Tutte polynomial of a matroid is determined by its \emph{configuration}: the abstract lattice $\mathcal Z(M)$ of cyclic flats together with the cardinality and rank of each cyclic flat~\cite{Ebe14}. Bonin and de Mier developed the structure theory of $\mathcal Z(M)$ systematically~\cite{BM08}. 

If $P$ is a finite poset, let $\mathcal C(P)$ denote the class of matroids for which $\mathcal Z(M)$ contains no induced subposet isomorphic to $P$. This  formulation is useful because the classes $\mathcal C(P)$ are minor-closed, see~\cref{lem:C(X)}. As a consequence of our work on $2$-inflated matroids and their boundaries, we obtain our main structural theorem in terms of the \emph{$2$-dimension} $\dim_2(P)$, i.e., the least $d$ for which $P$ is an induced subposet of the Boolean lattice $\mathcal B_d$; see~\cite{Tro75,Tro92}.

\begin{restatable}{thm}{boundedposet}\label{thm:bounded-poset-quotient}  
For every finite poset $P$ for $d:=\dim_2(P)$ it holds that:
\[
\Phi(\mathcal C(P))\leq 2^{4d-4}.
\]
\end{restatable}

Two important examples of a class $\mathcal C(P)$ come from bounding the complexity of a finite poset through its width or its height. For a poset $P$,  $w(P)$ denotes the maximum size of an antichain  and $h(P)$  the maximum size of a chain. Then, for $k\ge1$ we consider
\[
\mathrm{CW}(k)=\{M: w(\mathcal Z(M))\le k\},\qquad
\mathrm{CH}(k)=\{M: h(\mathcal Z(M))\le k\}.
\]
Thus, if $A_{k+1}$ is an antichain of size $k+1$ and $C_{k+1}$ a chain of size $k+1$, then $\mathrm{CW}(k)=\mathcal C(A_{k+1})$ and $\mathrm{CH}(k)=\mathcal C(C_{k+1})$. 

The cyclic-width hierarchy begins with nested matroids, equivalently Schubert or generalized Catalan matroids, at $\mathrm{CW}(1)$. The classes $\mathrm{CW}(2)$ and $\mathrm{CW}(3)$ were studied by Bonin and de Mier~\cite{BM08}. On the height side, the connected matroids in $\mathrm{CH}(2)$ are uniform, while the connected matroids in $\mathrm{CH}(3)$ are precisely the connected split matroids. Here, the multiplicative Merino-Welsh inequalties have been established by Ferroni and Schröter in~\cite{FS23}. Binary matroids of small cyclic height have also been studied from a structural point of view by Freil-Hollanti et al. in ~\cite{FGHW21}. Thus, bounded cyclic width and bounded cyclic height each extend previously studied classes and measure two complementary ways in which a cyclic-flat lattice can remain simple. We collect asymptotic consequences of \cref{thm:bounded-poset-quotient} for these classes as well as for posets of bounded order in \cref{cor:asymptotic-bounds}.

Using these results, we prove the multiplicative inequality for every matroid in $\mathrm{CW}(7)$, and hence for $\mathrm{CW}(k)$ for all $k\le7$:
\begin{restatable}{thm}{cyclicwidth}\label{thm:cyclicwidthseven}
The multiplicative Merino--Welsh inequality holds on $\mathrm{CW}(7)$.
\end{restatable}
The bounds obtained from the two boundary matroids reduce the remaining cases to finitely many decorated cores on at most nine elements. We then use the catalogue of matroids on at most nine elements to evaluate the required Tutte values on these cores rather than on the corresponding $2$-inflated matroids.

Also, we prove the multiplicative inequality for every matroid in $\mathrm{CH}(6)$: 
\begin{restatable}{thm}{cyclicheight}\label{thm:cyclicheightsix}
The multiplicative Merino--Welsh inequality holds on $\mathrm{CH}(6)$.
\end{restatable}

For cyclic height, the coarse boundary bound again leaves only finitely many parameter tuples. The mixed residual cases remain within the nine-element catalogue. The remaining one-sided cases correspond to $2$-thickenings $N^{(2)}$ of simple rank-five matroids $N$ on ten or eleven elements. The $2$-thickening identity reduces them to linear inequalities for the Tutte coefficients. These are established by linear programming and certified by exact rational Farkas certificates, so no catalogue of larger matroids is needed.

Cyclic height also connects directly with circuit size. Following Rajpal~\cite{Raj98}, a rank-$r$ matroid is \emph{$k$-paving} if every circuit has more than $r-k$ elements. Thus ordinary paving matroids are precisely the $1$-paving matroids, while $2$-, $3$-, and $4$-paving are equivalent to $g(M)\ge r(M)-1$, $g(M)\ge r(M)-2$, and $g(M)\ge r(M)-3$, respectively. We call $M$ $k$-copaving if $M^*$ is $k$-paving. The $k$-paving hierarchy has also been revisited recently by Singh~\cite{Sin25}. Paving and copaving matroids are already covered by the split-matroid theorem~\cite{FS23}; our height approach proves the exact multiplicative inequality through the next three levels, namely for $4$-paving and $4$-copaving matroids.

\begin{restatable}{thm}{fourpaving}\label{cor:four-paving}
Every loop- and coloop-free $4$-paving or $4$-copaving matroid satisfies the multiplicative Merino--Welsh inequality.
\end{restatable}

This is complementary to Csikv\'ari's recent theorem, which proves the Merino--Welsh inequality when, for some $\ell\ge4$, all circuits of both $M$ and $M^*$ have lengths between $\ell$ and $(\ell-2)^2(\ell^2-4\ell+2)$~\cite{Csi26}. His condition controls circuits and cocircuits simultaneously in an absolute range, whereas ours is one-sided and relative to rank. Even the following weakening of a conjecture of~\cite{Csi26} remains open.
\begin{conj}
 There exists a function $f$ such that $\Phi(M)\leq f(g)$ for every matroid $M$ of girth and cogirth $g\geq 3$.
\end{conj}

\paragraph{Organization} The paper is organized as follows. We establish the reduction to $2$-inflated matroids in \cref{sec:inflated} and the associated core and boundary matroids in \cref{sec:core}.  In \cref{sec:sparse-cyclic-flats} we develop the induced-poset framework for cyclic flats and use all of the above to prove our main structural result, \cref{thm:bounded-poset-quotient}. The cyclic-width section (\cref{sec:width}) uses the universal boundary estimate to isolate finitely many residual boundary pairs and proves the result through width seven by exact residual-core computations. The cyclic-height section (\cref{sec:height}) first derives a coarser boundary estimate and then pushes it through height six; its final one-sided cases are settled by the $2$-thickening identity and exact Tutte-coefficient certificates rather than by enlarging the catalogue. We then relate cyclic height to girth and cogirth deficiency in \cref{sec:deficiency} and close with a discussion of possible extensions in \cref{sec:discussion}.

\section{{2}-inflated matroids}\label{sec:inflated}

We call a loop- and coloop-free matroid \emph{inflated} if every element belongs to exactly one maximal nontrivial parallel or coparallel class. An inflated matroid is \emph{$2$-inflated} if all these classes have size $2$. 

We first record two standard ingredients. The parallel-class formula below is stated in the form in which it will be used.

\begin{lem}\label{lem:tutte-parallel}
Let $X$ be a parallel class of a matroid $M$, with $|X|=k\geq2$. If $X$ is not a cocircuit, then
\[
T_M(x,y)=T_{M\setminus X}(x,y)+(1+y+\cdots+y^{k-1})T_{M/X}(x,y).
\]
Dually, if $X$ is a coparallel class that is not a circuit, then
\[
T_M(x,y)=(1+x+\cdots+x^{k-1})T_{M\setminus X}(x,y)+T_{M/X}(x,y).
\]
\end{lem}
\begin{proof}
Retain one element $e$ of $X$ and call the resulting matroid $H_1$. Since $X$ is not a cocircuit, $e$ is not a coloop of $H_1$, and deletion--contraction gives $T_{H_1}=T_{M\setminus X}+T_{M/X}$. Reinsert the other elements of $X$ one at a time, and let $H_j$ have a parallel class of size $j$. Deleting the new element from $H_j$ gives $H_{j-1}$, while contracting it leaves the other $j-1$ members of the class as loops and otherwise gives $M/X$. Hence $T_{H_j}=T_{H_{j-1}}+y^{j-1}T_{M/X}$. Summing for $j=2,\ldots,k$ proves the first identity. The second is dual.
\end{proof}

We also use the activity expansion. For a linear ordering of $E(M)$ and a basis $B$, let $\iota_M(B)$ and $e_M(B)$ denote the numbers of internally and externally active elements, respectively. Then
\begin{equation}\label{eq:tutte-lin}
T_M(x,y)=\sum_{B\in\mathcal B(M)}x^{\iota_M(B)}y^{e_M(B)}.
\end{equation}
See, for example,~\cite[Section~7]{white1992matroid}.

\begin{lem}\label{lem:factorx}
Let $x,y\geq0$. If $e$ is not a loop of $M$, then $T_M(x,0)\geq xT_{M/e}(x,0)$. If $e$ is not a coloop of $M$, then $T_M(0,y)\geq yT_{M\setminus e}(0,y)$.
\end{lem}
\begin{proof}
It is enough to prove the first inequality, since the second follows by duality. Order $E(M)$ with $e$ first. By~\eqref{eq:tutte-lin}, every basis not containing $e$ has positive external activity and contributes zero to $T_M(x,0)$. If $e\in B$, then $e$ is internally active. Contracting $e$ decreases the internal activity of such a basis by $1$, while a basis of $M/e$ with zero external activity lifts to a basis of $M$ with zero external activity. Hence $T_M(x,0)\geq xT_{M/e}(x,0)$.
\end{proof}

\begin{lem}\label{prop:class2}
Let $M$ be loop- and coloop-free and let $X$ be a parallel or coparallel class of size $k\geq3$. Then $M$ has a proper loop- and coloop-free minor $N$ such that $\Phi(M)\leq\Phi(N)$.
\end{lem}
\begin{proof}
By duality it is enough to treat a parallel class $X$.

Suppose first that $X$ is a cocircuit. No circuit can meet both $X$ and $E(M)\setminus X$: such a circuit would meet the cocircuit $X$ in exactly one element, since a circuit containing two elements of the parallel class is the corresponding $2$-element circuit. Hence $M=M_0\oplus U_{1,k}$. Since $\Phi(U_{1,k})=k^2/(2(2^k-2))\leq1$ for $k\geq2$, replacing this component by $U_{1,2}$ produces a proper loop- and coloop-free minor with quotient at least $\Phi(M)$.

Assume now that $X$ is not a cocircuit. Delete all but one element $e\in X$, obtaining a matroid $H$, and put $P=H\setminus e=M\setminus X$ and $Q=H/e=M/X$. Write $a=T_P(1,1)$, $b=T_Q(1,1)$, $u=T_P(2,0)$, $v=T_Q(2,0)$, $c=T_P(0,2)$, and $d=T_Q(0,2)$. By~\cref{lem:tutte-parallel},
\begin{equation}\label{eq:parallel-k-values}
T_M(1,1)=a+kb,\qquad T_M(2,0)=u+v,\qquad T_M(0,2)=c+(2^k-1)d.
\end{equation}
The element $e$ is neither a loop nor a coloop of $H$. Thus~\cref{lem:factorx}, applied to $H$ at $(2,0)$ and $(0,2)$, gives $u\geq v$ and $d\geq c$.

First suppose $c>0$. The matroid $P$ is loop-free by construction and $Q$ is loop-free because $e$ has no parallel mate in $H$. Both matroids $P$ and $Q$ are coloop-free because  $d\geq c>0$. Put $R=\max\{\Phi(P),\Phi(Q)\}$. Then $a\leq\sqrt{Ruc}$ and $b\leq\sqrt{Rvd}$. It is therefore enough to prove $(\sqrt{uc}+k\sqrt{vd})^2\leq(u+v)(c+(2^k-1)d)$. Put $x=\sqrt{u/v}\geq1$ and $y=\sqrt{d/c}\geq1$. After division by $vc$, this becomes
\begin{equation}\label{eq:parallel-scalar}
(x+ky)^2\leq(x^2+1)(1+(2^k-1)y^2).
\end{equation}
The difference between the right- and left-hand sides is
$F_k(x,y)=(2^k-1)x^2y^2+(2^k-1-k^2)y^2-2kxy+1$. For $k\geq3$ and $x,y\geq1$, one has $\partial F_k/\partial x=2y((2^k-1)xy-k)>0$. After setting $x=1$, one has $\partial F_k/\partial y=2((2^{k+1}-2-k^2)y-k)\geq0$, since $2^{k+1}-2-k^2\geq k$ for $k\geq3$. Hence the minimum is $F_k(1,1)=2^{k+1}-(k+1)^2\geq0$. Thus~\eqref{eq:parallel-scalar} holds and~\eqref{eq:parallel-k-values} gives $\Phi(M)\leq R$.

Finally suppose $c=0$. Let $M_2$ be obtained from $M$ by retaining exactly two elements of $X$. First, $M_2$ is a proper minor of $M$, since $k\geq 3$ and at least one element of $X$ is deleted. It is loop- and coloop-free: deletion creates no loops, and any deleted element of $X$ may be replaced in a basis by a retained representative, so no new coloops arise. Since $X$ is not a cocircuit of $M$, the retained pair is not a cocircuit of $M_2$.
 Applying~\cref{lem:tutte-parallel} to $M$ and $M_2$ gives
\[
\frac{\Phi(M)}{\Phi(M_2)}=\left(\frac{a+kb}{a+2b}\right)^2\frac{3}{2^k-1}
\leq\frac{3k^2}{4(2^k-1)}<1
\]
for every $k\geq3$.
\end{proof}

\basislem*
\begin{proof}
It is enough to show that every loop- and coloop-free $M\in\mathcal M$ has a $2$-inflated minor $N$ with $\Phi(M)\leq\Phi(N)$. We prove this by induction on $|E(M)|$. If there is an element $e$ for which both $M\setminus e$ and $M/e$ are loop- and coloop-free, then $\Phi(M)\leq\max\{\Phi(M\setminus e),\Phi(M/e)\}$. Indeed, for $j\in\{0,1\}$ put $(a_j,b_j,c_j)=(T_{M_j}(2,0),T_{M_j}(0,2),T_{M_j}(1,1))$, where $M_0=M\setminus e$ and $M_1=M/e$. By deletion--contraction and the Cauchy--Schwarz inequality,
\[
\sqrt{\Phi(M)}\leq\frac{c_0+c_1}{\sqrt{a_0b_0}+\sqrt{a_1b_1}}
\leq\max_{j\in\{0,1\}}\frac{c_j}{\sqrt{a_jb_j}}.
\]
This is the factor-one case of the argument in~\cite[Lemma~4.3]{KMR18}. Hence one of the two proper minors has quotient at least $\Phi(M)$ and induction applies.

Otherwise, for every $e\in E(M)$ at least one of $M\setminus e$ and $M/e$ has a loop or coloop. Deletion cannot create loops and contraction cannot create coloops in a loop- and coloop-free matroid. Moreover, $M\setminus e$ has a coloop $f$ exactly when $\{e,f\}$ is a cocircuit of $M$, while $M/e$ has a loop $f$ exactly when $\{e,f\}$ is a circuit of $M$. Hence every element belongs to a nontrivial coparallel or parallel class. If an element belongs to both kinds of class, the corresponding $2$-element circuit and cocircuit coincide and form a direct-sum component $U_{1,2}$. Since $\Phi(U_{1,2})=1$ and $\Phi$ is multiplicative under direct sums, deleting this component gives a proper loop- and coloop-free minor with the same quotient, and induction applies. We may therefore assume that every element belongs to exactly one kind of class, so $M$ is inflated. If one of its maximal nontrivial parallel or coparallel classes has size at least $3$,~\cref{prop:class2} gives a proper loop- and coloop-free minor with quotient at least $\Phi(M)$, and induction applies. Otherwise all these classes have size $2$, so $M$ is $2$-inflated.
\end{proof}

\section{The core and its boundary matroids}\label{sec:core}

Let $M$ be an inflated matroid. Write $E(M)=A\mathbin{\dot\cup}B$, where $A$ is the union of the maximal nontrivial parallel classes and $B$ is the union of the maximal nontrivial coparallel classes. Choose one representative from each class and let $A'$ and $B'$ be the two sets of representatives. Define
$M'=M\setminus(A\setminus A')/(B\setminus B')$, on ground set $E'=A'\mathbin{\dot\cup}B'$. Thus the deletion on the $A$-side is the usual simplification restricted to the nontrivial parallel classes, while the contraction on the $B$-side is its dual, the corresponding cosimplification. Note that $M'$ is unique up to an isomorphism respecting the partition $A'\mathbin{\dot\cup}B'$.
We call $(M',A',B')$ the \emph{decorated core} of $M$, and call
$N=M'|A'$ and $K=(M'/A')^*$ its \emph{boundary matroids}. We record some elementary properties of the decorated core.

\begin{lem}\label{lem:boundary-simple}
The boundary matroids $N$ and $K$ are simple. Moreover, $r(N)=r_M(A)$ and $r(K)=r_{M^*}(B)$.
\end{lem}
\begin{proof}
The set $A'$ contains one representative from each maximal parallel class. Contracting all but one element of a coparallel class creates neither a loop nor a new parallel pair among the elements of $A'$: every circuit meeting a coparallel class contains that whole class. Hence $N$ is simple. The assertion for $K$ follows by duality.

For the rank identity, let $C=B\setminus B'$. If a circuit contained in $A\cup C$ met $C$, then, because it meets a coparallel class, it would contain the representative of that class in $B'$, a contradiction. Thus, $M|(A\cup C)= M|A\oplus M|C$. Also, $C$ is independent. 
Now, by definition of contraction, $r_{M/C}(A)=r_M(A\cup C)-r_M(C)=r_M(A)+r_M(C)-|C|=r_M(A)$. Deleting the extra parallel representatives does not change this rank, giving $r(N)=r_M(A)$. The formula for $K$ follows by duality.

\end{proof}

We will make use of the fact that the Tutte polynomial of an inflated matroid can be bounded in terms of its decorated core. For a loopless matroid $M$, let $\operatorname{si}(M)$ denote its simplification. We call $\operatorname{pe}(M)=|E(M)|-|E(\operatorname{si}(M))|$ the \emph{parallel excess} of $M$; equivalently, it is the number of elements in nontrivial parallel classes minus the number of such classes. For integers $r\geq2$ and $c\geq0$, put
\begin{equation}\label{eq:lambda}
\lambda(r,c)=(c+1)2^r-4c+2.
\end{equation}

\begin{lem}\label{lem:bjoerner}
Let $M$ be a connected matroid of rank $r\geq2$, and put $n=|E(M)|$. Then
$T_M(2,0)\geq\lambda(r,n-\operatorname{pe}(M)-r)$.
\end{lem}
\begin{proof}

Put $n'=|E(\operatorname{si}(M))|=n-\operatorname{pe}(M)$. Parallel extensions do not change $T(x,0)$, so $T_M(x,0)=T_{\operatorname{si}(M)}(x,0)$. Since $M$ is connected of rank at least $2$, so is $\operatorname{si}(M)$. By~\cite[Proposition~7.5.7]{white1992matroid}, we may write $T_{\operatorname{si}(M)}(x,0)=\sum_{i=0}^{r}h_i x^{r-i}$ with $h_0=1$, $h_i\geq n'-r$ for $1\leq i\leq r-2$, and $h_{r-1}\geq1$. Hence
$T_M(2,0)\geq2^r+(n'-r)(2^{r-1}+\cdots+2^2)+2=\lambda(r,n'-r)$.
\end{proof}

As a second ingredient, for a matroid $M$, let $I_j(M)$ denote the number of independent $j$-subsets of $E(M)$. Following Dowling~\cite{Dow80}, we call these the \emph{independent set numbers} of $M$. If $\mathcal I(M)$ denotes the independence complex, then $I_j(M)=f_{j-1}(\mathcal I(M))$. We will simply refer to $(I_j(M))_{0\le j\le r(M)}$ as the \emph{$f$-vector of $M$}.

\begin{lem}\label{lem:first-reduction-bound}
Let $M$ be $2$-inflated, $(M',A',B')$ its decorated core, and $N$ and $K$ its boundary matroids. Then
\begin{equation}\label{eq:basis-upper-bound}
T_M(1,1)\leq\sum_i2^{|B'|-r(M')+2i}I_i(N)I_{|B'|-r(M')+i}(K).
\end{equation}
Moreover, for rank and parallel excess we have:
\begin{equation}\label{eq:rank-formulas}
\begin{aligned}
r(M)&=r(M')+|B'|,& \operatorname{pe}(M)&=|A'|,\\
r^*(M)&=2|A'|+|B'|-r(M'),& \operatorname{pe}(M^*)&=|B'|.
\end{aligned}
\end{equation}
If $M$ is connected and $r(M),r^*(M)\geq2$, then
\begin{equation}\label{eq:bjoerner-bounds}
\begin{aligned}
T_M(2,0)&\geq\lambda\bigl(r(M),|A'|+|B'|-r(M')\bigr),\\
T_M(0,2)&\geq\lambda\bigl(r^*(M),r(M')\bigr).
\end{aligned}
\end{equation}
\end{lem}
\begin{proof}
Let $D$ be a basis of $M'$, put $X=D\cap A'$ and $Y=D\cap B'$, and let $J=B'\setminus Y$. Then $X$ is independent in $N$. Moreover, $D\subseteq A'\cup Y$ and $D$ spans $M'$, so $A'\cup Y$ spans $M'$. Hence $Y$ spans $M'/A'$, and therefore $J$ is independent in $K=(M'/A')^*$. Since $|Y|=r(M')-|X|$, one has $|J|=|B'|-r(M')+|X|$.

The basis $D$ lifts to exactly $2^{|X|+|J|}=2^{|B'|-r(M')+2|X|}$ bases of $M$. Summing over all bases of $M'$ and then forgetting the compatibility condition between $X$ and $J$ gives~\eqref{eq:basis-upper-bound}. Each coparallel class in the $2$-inflation increases rank by $1$, whereas a parallel class does not change rank, proving~\eqref{eq:rank-formulas}. Finally, substituting~\eqref{eq:rank-formulas} into~\cref{lem:bjoerner}, first for $M$ and then for $M^*$, gives~\eqref{eq:bjoerner-bounds}.
\end{proof}

It will be very useful to describe parts of the lattice of cyclic flats of an inflated matroid through its boundary matroids.

\begin{lem}\label{lem:subposet}
Let $M$ be inflated, $(M',A',B')$ its decorated core, and $N=M'|A'$ and $K=(M'/A')^*$ its boundary matroids. For $X\subseteq A'\mathbin{\dot\cup}B'$ denote by $\widehat X$ the union of classes indexed by elements of $X$.

The map $X\mapsto\widehat X$ reflects inclusion and maps every member of $\mathcal Z(M')$, every flat $X\in\mathcal F(N)$, and every set $A'\cup(B'\setminus F)$ with $F\in\mathcal F(K)$ to a cyclic flat of $M$. Hence these three families, with their actual comparabilities, form a single induced subposet of $\mathcal Z(M)$.

Moreover, if $B'=\emptyset$, then $\mathcal Z(M)\cong\mathcal F(N)$, and if $A'=\emptyset$, then $\mathcal Z(M)\cong\mathcal F(K)^*$.
\end{lem}

\begin{proof}
For $a\in A'$ and $b\in B'$, let $P_a$ and $S_b$ denote the corresponding parallel and coparallel classes of $M$.  
For a union $\widehat X$ of whole classes, the rank behavior under parallel and series extensions gives
$r_M(\widehat X)=r_{M'}(X)+\sum_{b\in X\cap B'}(|S_b|-1)$.

First let $X\in\mathcal Z(M')$. If $a\in A'\setminus X$, then $a\notin\operatorname{cl}_{M'}(X)$, so adjoining an element of $P_a$ to $\widehat X$ increases rank. If $b\in B'\setminus X$, adjoining one element of $S_b$ also increases rank by $1$. Thus $\widehat X$ is a flat of $M$. Every element in a selected class $P_a$ is non-coloop in $M|\widehat X$. If $b\in X\cap B'$, then $b$ is not a coloop of $M'|X$, so $r_{M'}(X)=r_{M'}(X\setminus\{b\})$; the same series-extension rank calculation shows that removing any one element of $S_b$ from $\widehat X$ changes the rank by $r_{M'}(X)-r_{M'}(X\setminus\{b\})=0$. Hence no element of a selected $S_b$ is a coloop either, and $\widehat X$ is cyclic.

Now let $X\in\mathcal F(N)$, so $X\subseteq A'$. The equality $\operatorname{cl}_{M'}(X)\cap A'=X$ shows exactly as above that $\widehat X$ is a flat of $M$, and it is cyclic because every one of its elements belongs to a nontrivial parallel class contained in $\widehat X$. This gives the copy of $\mathcal F(N)$. Applying the same argument to $M^*$ gives the upper family: under complementation, a flat $F$ of $K=(M'/A')^*$ corresponds to the cyclic flat indexed by $A'\cup(B'\setminus F)$.

Since the classes indexed by distinct elements of $E'$ are disjoint, $X\subseteq Y$ if and only if $\widehat X\subseteq\widehat Y$. Thus all comparabilities between the three families are preserved and reflected, so their union is an induced subposet of $\mathcal Z(M)$.

Finally, if $B'=\emptyset$, every flat of $M$ is a union of whole parallel classes and every such flat is cyclic. Hence the preceding correspondence gives $\mathcal Z(M)\cong\mathcal F(N)$. The case $A'=\emptyset$ follows dually.
\end{proof}

\section{Sparse lattices of cyclic flats}\label{sec:sparse-cyclic-flats}

We now introduce the main type of matroid class for the rest of the paper. 

\begin{defi}
Given a finite poset $P$, let $\mathcal C(P)$ be the class of matroids $M$ for which $\mathcal Z(M)$ does not contain an induced subposet isomorphic to $P$.
\end{defi}
A basic observation, going back to~\cite{Bry75,Ing77}, is the following.
\begin{obs}\label{obs:dual}
For every matroid $M$ on ground set $E$, complementation induces an order-isomorphism $\mathcal Z(M^*)\cong\mathcal Z(M)^*$: a cyclic flat $X$ of $M^*$ corresponds to the cyclic flat $E\setminus X$ of $M$.
\end{obs}

We first record the corresponding minor-closure statement.

\begin{lem}\label{lem:C(X)}
For every finite poset $P$, the class $\mathcal C(P)$ is minor-closed. If $P\cong P^*$, then $\mathcal C(P)$ is also closed under duality.
\end{lem}
\begin{proof}
It suffices to show that the cyclic-flat poset of a single-element deletion or contraction embeds as an induced subposet of $\mathcal Z(M)$. Let $e\in E(M)$ and $X\in\mathcal Z(M\setminus e)$. Since $X$ is a flat of $M\setminus e$, its closure in $M$ is either $X$ or $X\cup\{e\}$. In the first case $X$ is already a cyclic flat of $M$. In the second case $e\in\operatorname{cl}_M(X)$, so a circuit $C$ of $M$ satisfies $e\in C\subseteq X\cup\{e\}$; because $X$ is cyclic in $M\setminus e$, this shows that $X\cup\{e\}$ is cyclic as well. Thus
$\varphi(X)=\operatorname{cl}_M(X)$ belongs to $\mathcal Z(M)$.

Moreover, $\varphi(X)\cap(E(M)\setminus\{e\})=X$. Therefore $\varphi$ is injective and, for $X,Y\in\mathcal Z(M\setminus e)$, one has $\varphi(X)\subseteq\varphi(Y)$ if and only if $X\subseteq Y$. Hence $\mathcal Z(M\setminus e)$ is isomorphic to an induced subposet of $\mathcal Z(M)$. Applying the deletion statement to $M^*$ and using~\cref{obs:dual} gives the corresponding assertion for $M/e$. Thus $\mathcal C(P)$ is minor-closed. Finally, if $P\cong P^*$, then~\cref{obs:dual} also gives closure under duality.
\end{proof}

The preceding estimates also give a quantitative bound. 

\boundedposet*
\begin{proof}
By~\Cref{lem:C(X),lem:inductionbasis}, it is enough to consider a
$2$-inflated $M\in\mathcal C(P)$. Let $(M',A',B')$ be its decorated
core, with boundary matroids $N$ and $K$.

If a matroid has rank $a$, the closures of the subsets of a basis form
an induced copy of $\mathcal B_a$ in its lattice of flats. By
\cref{lem:subposet}, both $\mathcal F(N)$ and $\mathcal F(K)^*$ occur
as induced subposets of $\mathcal Z(M)$. Hence $r(N),r(K)<d$.

Put $a=r(N)$, $b=r(K)$, $p=|A'|$, $q=|B'|$, and $s=p+q$. In a
nonzero summand of~\eqref{eq:basis-upper-bound}, put
$j=b-a+i$. Then $i\leq a$, $j\leq b$, and
$b-a+2i=i+j\leq a+b$. Therefore

\[
T_M(1,1)
\leq 2^{a+b}\sum_i\binom pi\binom q{b-a+i}
\leq 2^{a+b+s}.
\]

Every loop- and coloop-free matroid $L$ satisfies
\[
T_L(2,0)T_L(0,2)\geq2^{|E(L)|}.
\]
For each connected component with rank and corank at least $2$, apply
\cref{lem:bjoerner} to the component and its dual. The rank-$1$ and
corank-$1$ cases are immediate, and the general case follows by
multiplicativity. Since $|E(M)|=2s$,

\[
\Phi(M)\leq
\frac{2^{2(a+b+s)}}{4^s}
=2^{2(a+b)}
\leq 2^{4d-4}.
\]

\end{proof}

Recall that, if $A_{k+1}$ is an antichain of size $k+1$ and $C_{k+1}$ a chain of size $k+1$, then $\mathrm{CW}(k)=\mathcal C(A_{k+1})$ and $\mathrm{CH}(k)=\mathcal C(C_{k+1})$. Further, for an integer $p$, define
\[
\Phi(p)=\max\{\Phi(\mathcal C(P))\mid |P|=p\}.
\]
Then:

\begin{cor}\label{cor:asymptotic-bounds}
We have

\[
\Phi(p)=2^{\Theta(p)},\qquad
\Phi(\mathrm{CH}(k))=2^{\Theta(k)},\qquad
\Phi(\mathrm{CW}(k))=k^{\Theta(1)}.
\]

\end{cor}

\begin{proof}
The upper bounds follow from~\cref{thm:bounded-poset-quotient}. Indeed,
$\dim_2(P)\leq |P|$, $\dim_2(C_{k+1})=k$, and, by Sperner's theorem,
$\dim_2(A_{k+1})=\Theta(\log k)$. Hence

\[
\Phi(p)=2^{O(p)},\qquad
\Phi(\mathrm{CH}(k))=2^{O(k)},\qquad
\Phi(\mathrm{CW}(k))=k^{O(1)}.
\]

For the lower bounds, let $M_n=U^{(2)}_{2n/3,n}$ with $3\mid n$. By~\cite[Lemma~2.4]{BCCP24false},
\[
\Phi(M_n)=2^{\Theta(n)}.
\]
Moreover,

\[
\mathcal Z(M_n)\cong
\{X\subseteq[n]:|X|\leq2n/3-1\}\cup\{\hat1\}.
\]

Thus its height is $\Theta(n)$, it excludes a chain of order
$\Theta(n)$, and its width is
$\binom n{\lfloor n/2\rfloor}=2^{\Theta(n)}$.
Consequently,

\[
\Phi(p)=2^{\Omega(p)},\qquad
\Phi(\mathrm{CH}(k))=2^{\Omega(k)},\qquad
\Phi(\mathrm{CW}(k))=k^{\Omega(1)}.
\]

For each sufficiently large $p$, choose $n$ divisible by $3$ with $2n/3+1<p$ and $n=\Theta(p)$; then $M_n\in\mathcal C(C_p)$. Likewise, for each sufficiently large $k$, choose $n$ divisible by $3$ with $2n/3+1\leq k$ and $n=\Theta(k)$; then $M_n\in\mathrm{CH}(k)$. For cyclic width, choose the largest admissible $n$ with $\binom n{\lfloor n/2\rfloor}\leq k$; consecutive admissible values differ only by a bounded factor, so $n=\Theta(\log k)$. Together with the upper bounds, this proves the result.
\end{proof}

The classes $\mathcal C(P)$ are restrictive. For example, a class of loopless matroids that contains members of arbitrarily large rank and is closed under parallel extensions cannot be contained in any fixed $\mathcal C(P)$. Indeed, replace every element of a loopless matroid $Q$ by a parallel pair. For each flat $F$ of $Q$, the union of the pairs indexed by $F$ is a cyclic flat of this $2$-thickening, and these cyclic flats form an induced copy of $\mathcal F(Q)$. The closures of the subsets of a basis of $Q$ then give Boolean lattices of arbitrarily large dimension.

\section{Matroids of small width}\label{sec:width}
The general asymptotic bounds for $\mathrm{CW}(k)$ are given in~\cref{cor:asymptotic-bounds}. We now obtain sharper results for small cyclic width.

\begin{lem}\label{lem:width}
Let $M$ be inflated, with decorated core $(M',A',B')$ and boundary matroids $N=M'|A'$ and $K=(M'/A')^*$. Then
\[
w(\mathcal Z(M))\geq w(\mathcal Z(M')).
\]
If $A'\neq\emptyset$, then
\[
w(\mathcal Z(M))\geq
w(\mathcal F(N))+
w\bigl(\mathcal Z(M')_{\subseteq B'}\setminus\{\emptyset\}\bigr),
\]
and if $B'\neq\emptyset$, then
\[
w(\mathcal Z(M))\geq
w(\mathcal F(K))+
w\bigl(\mathcal Z(M')_{\supseteq B'}\setminus\{E'\}\bigr).
\]
\end{lem}
\begin{proof}
The first inequality follows from the induced copy of $\mathcal Z(M')$ in~\cref{lem:subposet}.

Suppose that $A'\neq\emptyset$. Choose a maximum antichain $\mathcal A$ of $\mathcal F(N)$ that does not contain $\emptyset$. Such a choice is always possible: if $w(\mathcal F(N))=1$, take $\mathcal A=\{A'\}$, while if the width is at least $2$, no maximum antichain can contain the least element. Let $\mathcal B$ be a maximum antichain of $\mathcal Z(M')_{\subseteq B'}\setminus\{\emptyset\}$. Under the joint embedding of~\cref{lem:subposet}, every member of $\mathcal A$ is indexed by a nonempty subset of $A'$, whereas every member of $\mathcal B$ is indexed by a nonempty subset of $B'$. Since $A'\cap B'=\emptyset$, no member of one family is comparable with a member of the other. Thus their images form an antichain of size $|\mathcal A|+|\mathcal B|$, proving the second inequality.

The third inequality is dual. Equivalently, choose a maximum antichain in the upper copy of $\mathcal F(K)^*$ that avoids its greatest element $E'$, and combine it with a maximum antichain of $\mathcal Z(M')_{\supseteq B'}\setminus\{E'\}$. A member of the first family contains all of $A'$ but not all of $B'$, while a member of the second contains all of $B'$ but is not $E'$, so the two families are cross-incomparable.
\end{proof}

For later use, let $M$ be a connected $2$-inflated matroid whose rank and corank are at least $2$. Let $N$ and $K$ be its boundary matroids. The rank of the core is then forced by $r(M')=r(N)+|E(K)|-r(K)$. Combining~\cref{lem:first-reduction-bound} with~\eqref{eq:lambda} gives the universal boundary estimate
\begin{equation}\label{eq:cw4-bound}
\begin{aligned}
\Phi(M)
&\leq
\frac{\left(\sum_i2^{|E(K)|-r(M')+2i}I_i(N)I_{|E(K)|-r(M')+i}(K)\right)^2}
{\lambda(r(M')+|E(K)|,|E(N)|+|E(K)|-r(M'))}\\
&\qquad{}\cdot
\frac{1}{\lambda(2|E(N)|+|E(K)|-r(M'),r(M'))}.
\end{aligned}
\end{equation}

\begin{prop}\label{thm:cyclicwidthfour}
The multiplicative Merino--Welsh inequality holds on $\mathrm{CW}(4)$.
\end{prop}
\begin{proof}
The Tutte polynomial is multiplicative under direct sums, while the rank $1$ and corank $1$ cases are immediate. Since $\mathrm{CW}(4)$ is minor-closed by~\cref{lem:C(X)}, \cref{lem:inductionbasis} reduces the proof to a connected $2$-inflated matroid $M$. Let $N$ and $K$ be the boundary matroids of its decorated core.

By~\cref{lem:width}, both boundary lattices of flats have width at most $4$. Since $N$ and $K$ are simple by~\cref{lem:boundary-simple}, their singleton flats form antichains. Hence $|E(N)|,|E(K)|\leq4$. The possible boundary matroids are
\[
U_{0,0},\ U_{1,1},\ U_{2,2},\ U_{2,3},\ U_{3,3},\ U_{2,4},\ U_{2,3}\oplus U_{1,1}.
\]
Their $f$-vectors, in the same order, are $(1)$, $(1,1)$, $(1,2,1)$, $(1,3,3)$, $(1,3,3,1)$, $(1,4,6)$, and $(1,4,6,3)$.
Indeed, the assertion is immediate on at most $3$ elements. On four elements, a simple rank-three matroid either has a triangle, in which case it is $U_{2,3}\oplus U_{1,1}$, or every pair is a rank $2$ flat, giving six incomparable rank $2$ flats. The free rank $4$ matroid also has six rank $2$ flats. Thus no other simple matroid on four elements has flat-lattice width at most $4$.

For each of the $7^2=49$ ordered pairs $(N,K)$, substitute the forced value of $r(M')$ in~\eqref{eq:cw4-bound}. After discarding rank $1$ and corank $1$ cases, exact arithmetic gives
\begin{center}
\begin{tabular}{c|c}
$|E'|$ & maximum in~\eqref{eq:cw4-bound} \\ \hline
$2$ & $25/36$\\
$3$ & $4/7$\\
$4$ & $121/196$\\
$5$ & $1849/2655$\\
$6$ & $60025/60516$\\
$7$ & $279841/310284$\\
$8$ & $1481089/1602756$
\end{tabular}
\end{center}
All values are smaller than $1$. The global maximum occurs for $N=K=U_{3,3}$. Its $f$-vector is $(1,3,3,1)$. After the inflation factors in~\eqref{eq:basis-upper-bound}, the numerator before squaring is $1+6^2+12^2+8^2=245$, while both denominator factors are $\lambda(6,3)=246$. Hence $\Phi(M)\leq(245/246)^2<1$.
\end{proof}

\begin{prop}\label{thm:cyclicwidthfive}
The multiplicative Merino--Welsh inequality holds on $\mathrm{CW}(5)$.
\end{prop}
\begin{proof}
Proceed as in~\cref{thm:cyclicwidthfour}. Now~\cref{lem:width} gives
\[
|E(N)|,|E(K)|\leq5,
\qquad
w(\mathcal F(N)),w(\mathcal F(K))\leq5.
\]
Besides the seven boundary matroids occurring for cyclic width $4$, the only new possibilities are
\[
U_{2,5}\qquad\text{and}\qquad U_{2,4}\oplus U_{1,1}.
\]
Their $f$-vectors are $(1,5,10)$ and $(1,5,10,6)$, respectively.
To see this, consider a simple matroid on $5$ elements. In rank $3$, the rank $2$ flats partition the $10$ pairs. If there is no $4$-point line, at most two $3$-point lines can occur, since $3$ distinct $3$-point lines would require at least six points. Hence there are at least six rank $2$ flats. A $4$-point line gives $U_{2,4}\oplus U_{1,1}$. In rank $4$, every rank $2$ flat has size at most $3$, and the same pair count again gives at least six rank $2$ flats. Ranks $2$ and $5$ give $U_{2,5}$ and the free matroid, respectively, and the latter has ten rank $2$ flats.

Thus only $9^2=81$ ordered boundary pairs occur. The $49$ pairs not involving a $5$-element boundary matroid were already checked in~\cref{thm:cyclicwidthfour}. Exact substitution in~\eqref{eq:cw4-bound} for the remaining $32$ pairs gives $534361/861212<1$ as the maximum. Up to duality it is attained for $N=U_{2,3}\oplus U_{1,1}$ and $K=U_{2,4}\oplus U_{1,1}$. The numerator before squaring is $2193$, while the denominator factors are $5106$ and $1518$. Hence every case satisfies the multiplicative inequality.
\end{proof}

The preceding two theorems remove all cyclic widths at most $5$ from any subsequent computation. We now use the same boundary estimate to isolate the remaining cases rather than as a final inequality. For boundary matroids $N$ and $K$, let $B(N,K)$ denote the right-hand side of~\eqref{eq:cw4-bound}, after substituting the forced core rank
\[
r(M')=r(N)+|E(K)|-r(K).
\]
For a boundary matroid $Q$, define its \emph{boundary data} by
\begin{equation}\label{eq:boundary-data}
\sigma(Q)=\bigl(|E(Q)|,r(Q),(I_j(Q))_{0\le j\le r(Q)}\bigr).
\end{equation}
Hence $B(N,K)$ depends only on $\sigma(N)$ and $\sigma(K)$. For width $k$, consider simple boundary matroids $Q$ with $w(\mathcal F(Q))\le k$. Let $\mathcal R_k$ be the set of ordered pairs of boundary data $(\sigma(N),\sigma(K))$ arising from such $N,K$, after discarding rank-$1$ and corank-$1$ cases and pairs with $B(N,K)\le1$. We call the members of $\mathcal R_k$ the \emph{residual boundary pairs}.

For $k=7$ the set $\mathcal R_7$ consists of $14$ residual boundary pairs. Their total boundary sizes, equivalently the possible core sizes, are $7$, $8$, and $9$, with multiplicities $4$, $8$, and $2$, respectively. In particular, every residual core lies within the nine-element range of the Mayhew--Royle catalogue~\cite{MR08,MR22data}.


Direct computation of the Tutte polynomial of a $2$-inflation may involve up to $18$ elements. The following formular evaluate the three required values directly from the core.

\begin{prop}
Let \(M_2\) be the 2-inflation of \(M\) with respect to
$E(M)=A\mathbin{\dot\cup}B$. Then
\[
T_{M_2}(1,1)
=
\sum_{X\in\mathcal B(M)}
2^{\,|X\cap A|+|B\setminus X|},
\]
\[
T_{M_2}(2,0)
=
\sum_{X\subseteq E(M)}
(-1)^{r_M(X)+|X|}
3^{\,|B\setminus X|},
\]
and
\[
T_{M_2}(0,2)
=
\sum_{X\subseteq E(M)}
(-1)^{r(M)-r_M(X)}
3^{\,|X\cap A|}.
\]
\end{prop}

\begin{proof}
The first equality follows by counting the lifts of each basis of $M$, and the third follows from the second by duality. Hence, we focus on the second one. 

We first observe that adding a parallel element does not
change the evaluation at \((2,0)\). 
Therefore the parallel extensions corresponding to the elements of
\(A\) may be ignored when evaluating at \((2,0)\).

Now suppose that \(b'\) is added in series with \(b\). Then
deletion--contraction gives
\[
T_{M'}(x,y)
=
T_M(x,y)+x\,T_{M\setminus b}(x,y).
\]
Indeed, contracting \(b'\) recovers \(M\), while deleting \(b'\) makes
\(b\) a coloop. Consequently,
\[
T_{M'}(2,0)
=
T_M(2,0)+2T_{M\setminus b}(2,0).
\]

Applying this successively to all elements of \(B\), we obtain
\[
T_{M_2}(2,0)
=
\sum_{C\subseteq B}
2^{|C|}T_{M\setminus C}(2,0).
\tag{1}
\]

We now use the subset expansion of the Tutte polynomial:
\[
T_N(x,y)
=
\sum_{X\subseteq E(N)}
(x-1)^{r(N)-r_N(X)}
(y-1)^{|X|-r_N(X)}.
\]
At \((2,0)\), this becomes
\[
T_N(2,0)
=
\sum_{X\subseteq E(N)}
(-1)^{|X|-r_N(X)}.
\]
Since $(-1)^{|X|-r_N(X)}=(-1)^{|X|+r_N(X)}$, 
for \(N=M\setminus C\) we have
\[
T_{M\setminus C}(2,0)
=
\sum_{X\subseteq E(M)\setminus C}
(-1)^{|X|+r_M(X)}.
\]

Substituting this into \((1)\) gives
\[
T_{M_2}(2,0)
=
\sum_{C\subseteq B}
2^{|C|}
\sum_{X\subseteq E(M)\setminus C}
(-1)^{|X|+r_M(X)}.
\]
Interchanging the order of summation, for a fixed
\(X\subseteq E(M)\) the admissible sets \(C\) are precisely those
satisfying $C\subseteq B\setminus X$.

Hence
\[
T_{M_2}(2,0)
=
\sum_{X\subseteq E(M)}
(-1)^{|X|+r_M(X)}
\sum_{C\subseteq B\setminus X}2^{|C|}.
\]
Finally,
\[
\sum_{C\subseteq B\setminus X}2^{|C|}
=
(1+2)^{|B\setminus X|}
=
3^{|B\setminus X|}.
\]

\end{proof}

For the cyclic-width test, it is enough to inspect the unions of whole classes. If $X\subseteq E'$ and $\widehat X$ is the union of the classes indexed by $X$, then $\widehat X$ is a cyclic flat exactly when
\[
\operatorname{cl}_{M'}(X)\cap A'=X\cap A'
\quad\text{and}\quad
r_{M'}(X)=r_{M'}(X\setminus\{b\})\ \text{for every }b\in X\cap B'.
\]

\begin{rem}[Exact evaluation of a decorated core]\label{rem:cw-core-computation}
 The exact cyclic width and the three Tutte evaluations of a $2$-inflated matroid can be computed directly from the rank function of $M'$, without explicitly constructing a $2|E(M')|$-element matroid.

\end{rem}

\cyclicwidth*
\begin{proof}
Let $M\in\mathrm{CW}(7)$ be a minimum counterexample. By multiplicativity and~\cref{lem:inductionbasis}, we may assume that $M$ is connected and $2$-inflated. There are only the $14$ residual boundary pairs described above, so $|E(M')|\in\{7,8,9\}$.

We enumerate the compatible decorated cores with these residual boundaries, compute their exact cyclic width by~\cref{rem:cw-core-computation}, and retain only expansions of cyclic width six or seven. Exact evaluation of $T_M(1,1)$, $T_M(2,0)$, and $T_M(0,2)$ gives no failure of the multiplicative inequality. As a consistency check independent of the downloaded catalogue, a local orderly-generation implementation reconstructs the required matroids through nine elements and agrees with the Mayhew--Royle data~\cite{MR08,MR22data}. Since widths at most five are already covered by~\cref{thm:cyclicwidthfive}, no minimum counterexample exists.
\end{proof}
The computations used in the cyclic-width \(6\) and \(7\) cases were performed in SageMath 10.6. The source code, original exploratory notebooks, and reproducibility scripts are archived in~\cite{Merino2026CyclicWidth}.

The residual boundary pairs also make clear what would have to be strengthened in order to treat larger cyclic widths: beyond the range considered here, the boundary reduction need not close within the available small-matroid catalogue, so either additional catalogue data or further structural input may be required. We return to this point in~\cref{sec:discussion}.

\section{Matroids of small height}\label{sec:height}
The general asymptotic bounds for $\mathrm{CH}(k)$ are given in~\cref{cor:asymptotic-bounds}. We now obtain sharper results for small cyclic height.

\begin{lem}\label{lem:height-first-reduction}
Let $M$ be inflated, with decorated core $(M',A',B')$ and boundary matroids $N$ and $K$. Then
\[
h(\mathcal Z(M))\geq\max\{h(\mathcal Z(M')),r(N)+r(K)+1\}.
\]
\end{lem}
\begin{proof}
The first inequality follows from the induced copy of $\mathcal Z(M')$ in~\cref{lem:subposet}. For the second, take a maximal chain in $\mathcal F(N)$ from $\emptyset$ to $A'$ and concatenate it at $A'$ with a maximal chain in the upper copy of $\mathcal F(K)^*$ from $A'$ to $E'$. By the joint induced embedding in~\cref{lem:subposet}, this is a chain of cyclic flats of $M$. The two chains have respectively $r(N)+1$ and $r(K)+1$ elements and meet only at $A'$, giving $r(N)+r(K)+1$ elements in total.
\end{proof}

\begin{lem}\label{lem:height-bound}
Let $M$ be a connected $2$-inflated matroid, with decorated core $(M',A',B')$ and boundary matroids $N$ and $K$. If $r(M),r^*(M)\geq2$, then
\begin{equation}\label{eq:height-bound}
\begin{aligned}
\Phi(M)
&\leq
\frac{\left(\displaystyle\sum_{i=\max\{0,r(N)-r(K)\}}^{r(N)}
2^{r(K)-r(N)+2i}
\binom{|A'|}{i}
\binom{|B'|}{r(K)-r(N)+i}\right)^2}
{\lambda(r(N)+2|B'|-r(K),|A'|-r(N)+r(K))}\\
&\qquad{}\cdot
\frac{1}{\lambda(2|A'|-r(N)+r(K),|B'|+r(N)-r(K))}.
\end{aligned}
\end{equation}
\end{lem}
\begin{proof}
By~\cref{lem:boundary-simple}, the boundary matroids are simple, so $I_j(N)\leq\binom{|A'|}{j}$ and $I_j(K)\leq\binom{|B'|}{j}$. Moreover, the forced core-rank identity is $r(M')=r(N)+|B'|-r(K)$. Substitution in~\eqref{eq:basis-upper-bound} and~\eqref{eq:bjoerner-bounds} gives~\eqref{eq:height-bound}.
\end{proof}

\begin{prop}\label{thm:cyclicheight}
The multiplicative Merino--Welsh inequality holds on $\mathrm{CH}(4)$.
\end{prop}
\begin{proof}
By multiplicativity and the rank $1$ and corank $1$ cases, it is enough to consider connected matroids of rank and corank at least $2$. By~\Cref{lem:C(X),lem:inductionbasis}, we may further assume that $M$ is $2$-inflated. For the numerical case analysis only, abbreviate $a=r(N)$, $b=r(K)$, $p=|A'|$, and $q=|B'|$. By~\cref{lem:height-first-reduction}, $a+b\leq3$. Simplicity gives $p=0$ if $a=0$, $p=1$ if $a=1$, and $p\geq a$ if $a\geq2$, and similarly for $q,b$. Up to duality, the nontrivial rank pairs are $(3,0),(2,1),(2,0),(1,1)$.

For $(a,b)=(3,0)$,~\eqref{eq:height-bound} becomes
\[
\frac{64\binom p3^2}{(4p-2)(2^{2p-1}-10)}.
\]
Its maximum for $p\geq3$ is $1600/2259$, attained at $p=5$; comparison of consecutive terms shows that the sequence decreases for $p\geq5$. For $(a,b)=(2,0)$ the bound is $2p^2(p-1)^2/[9(4^{p-1}-2)]$, whose maximum is $4/7$ at $p=3$. For $(a,b)=(2,1)$, simplicity forces $q=1$, and the bound is $4p^2(2p-1)^2/[3(4p+6)(2^{2p-1}-2)]$, whose maximum is $4/7$ at $p=2$. Finally, $(a,b)=(1,1)$ forces $p=q=1$ and gives $25/36$. Every case is strictly below $1$, with global maximum $1600/2259$.
\end{proof}

\begin{prop}\label{thm:cyclicheightfive}
The multiplicative Merino--Welsh inequality holds on $\mathrm{CH}(5)$.
\end{prop}
\begin{proof}
By~\Cref{lem:C(X),lem:inductionbasis}, it is enough to consider a connected $2$-inflated matroid. Abbreviate $a=r(N)$, $b=r(K)$, $p=|A'|$, and $q=|B'|$. By~\cref{lem:height-first-reduction}, $a+b\leq4$. Here the already proved result~\cref{thm:cyclicheight} disposes of every case with $a+b\leq3$. Thus, up to duality, the only new rank pairs are $(4,0),(3,1),(2,2)$.

For $(a,b)=(4,0)$,~\eqref{eq:height-bound} becomes
\[
\frac{256\binom p4^2}{(12p-30)(5\cdot2^{2p-4}-14)}.
\]
Direct comparison of consecutive terms shows that the only values at least $1$ are $p=6$ and $p=7$, where the bounds are $1600/1477$ and $78400/68931$, respectively. For $(a,b)=(3,1)$, simplicity gives $q=1$, and the bound is
\[
\frac{\left(4\binom p2+16\binom p3\right)^2}{(12p-6)(2^{2p}-10)},
\]
whose maximum is $1936/2583<1$ at $p=4$.

For $(a,b)=(2,2)$,~\eqref{eq:height-bound} is
\[
\frac{\left(1+4pq+16\binom p2\binom q2\right)^2}{\lambda(2q,p)\lambda(2p,q)}.
\]
Assume by duality that $p\geq q$. If $q\geq5$, then $1+4pq+16\binom p2\binom q2<4p^2q^2$, while $\lambda(2q,p)>p(4^q-4)$ and $\lambda(2p,q)>q(4^p-4)$. Since $4^m-4>4m^3$ for $m\geq5$, the quotient is strictly smaller than $1$. For $q=2,3,4$, substituting the fixed value of $q$ and comparing consecutive terms gives a decreasing sequence for $p\geq q$; the largest of these sequences is $121/196<1$, attained at $p=q=2$. Consequently only the one-sided cases $(a,b,p,q)=(4,0,6,0)$ and $(4,0,7,0)$, together with their duals, remain.

In these cases $B'=\emptyset$, so $M'=N$. By~\cref{lem:subposet}, the cyclic-flat poset of the $2$-thickening is the lattice of flats $\mathcal F(N)$, which has height $5$. Let $N^{(2)}$ denote the matroid obtained by replacing each element of $N$ by a parallel pair. Grouping subsets of $E(N^{(2)})$ according to their support in $E(N)$ in the rank-sum definition of the Tutte polynomial gives
\begin{equation}\label{eq:thickening-tutte}
T_{N^{(2)}}(x,y)=(1+y)^{r(N)}T_N\left(1+\frac{x-1}{1+y},y^2\right).
\end{equation}
Since $r(N)=4$, the multiplicative inequality for $N^{(2)}$ is equivalent to
\begin{equation}\label{eq:ch5-core-check}
256T_N(1,1)^2\leq81T_N(2,0)T_N(2/3,4).
\end{equation}
It remains to check~\eqref{eq:ch5-core-check} for simple connected rank-$4$ matroids on six or seven elements. We reconstructed these matroids independently: there are $6$ isomorphism classes on six elements and $39$ on seven elements, in agreement with the Mayhew--Royle catalogue~\cite{MR08,MR22data}. Exact evaluation of the three Tutte values verifies~\eqref{eq:ch5-core-check} in every case. This completes the proof.
\end{proof}

\begin{rem}[The finite verification]\label{rem:ch5-verification}
The three evaluations in~\eqref{eq:ch5-core-check} can be accumulated in a single pass over the $2^n$ subsets of $E(N)$, with $n\in\{6,7\}$. Among the $45$ matroids checked, the largest value of
\[
\frac{256T_N(1,1)^2}{81T_N(2,0)T_N(2/3,4)}
\]
is $720/1937\approx0.371709$, attained by $U_{4,6}$. For this matroid,
\[
T_N(1,1)=15,\qquad T_N(2,0)=52,\qquad T_N(2/3,4)=\frac{2980}{81},
\]
so the required comparison is $57600<154960$.
\end{rem}

The proof above show that for height and for fixed boundary ranks $(a,b)$ the numerator in~\eqref{eq:height-bound} grows only polynomially in the boundary sizes, while the denominator grows exponentially. Hence
at every fixed cyclic height, only finitely many boundary-size parameter tuples remain. For cyclic height six, most of the remaining cases are deal with as above, but for some final cases we need to study a system of linear inequalities on the coefficients of the Tutte polynomial.

\cyclicheight*
\begin{proof}
Let $M\in\mathrm{CH}(6)$ be a minimum counterexample. As before, we may assume that $M$ is connected and $2$-inflated. Put
\[
a=r(N),\qquad b=r(K),\qquad p=|A'|,\qquad q=|B'|.
\]
By~\cref{lem:height-first-reduction}, $a+b\le5$, while~\cref{thm:cyclicheightfive} settles every case with $a+b\le4$. Thus, up to duality, only the rank pairs
\[
(5,0),\qquad(4,1),\qquad(3,2)
\]
are new. Substitution in~\eqref{eq:height-bound} eliminates every parameter tuple with $(a,b)=(3,2)$. For the other two rank pairs, the only parameter tuples not ruled out by the coarse bound are
\begin{equation}\label{eq:ch6-residual-parameters}
(5,0,p,0),\quad 7\le p\le11,
\qquad\text{and}\qquad
(4,1,p,1),\quad5\le p\le8.
\end{equation}
For $p\ge12$ in the $(5,0,p,0)$ case the same bound is already strictly below one.

First consider $(4,1,p,1)$. Restoring the actual independent-set numbers in the universal boundary estimate leaves $21$ boundary matroids. Every compatible core has at most nine elements. After an exact catalogue study and cyclic-height test, the  remaining number of decorated cores are $3294$ on $6,7,8$ or $9$ elements. All of them satisfy the inequality.

It remains to treat the $(5,0,p,0)$ cases. Here $B'=\emptyset$, so $M'=N$ and $M=N^{(2)}$ is the $2$-thickening of a simple connected rank-five matroid $N$. For $p=7,8,9$, a catalogue search gives  respectively $22$, $217$, and $188936$ cases, all of which satisfy the multiplicative inequality. Thus only $p=10,11$ lie beyond the available nine-element catalog.

For these two sizes the $2$-thickening identity~\eqref{eq:thickening-tutte} gives
\[
T_M(1,1)=32T_N(1,1),\qquad
T_M(2,0)=T_N(2,0),\qquad
T_M(0,2)=243T_N(2/3,4).
\]
Since $N$ is simple of rank five, $g(N)\ge3$. The rank-chain bound of ~\cref{thm:large-girth} therefore gives $h(\mathcal Z(N))\le5$, so~\cref{thm:cyclicheightfive} gives
\begin{equation}\label{eq:ch6-rank5-mw}
T_N(1,1)^2\le T_N(2,0)T_N(0,2).
\end{equation}
It is therefore enough to prove
\begin{equation}\label{eq:ch6-auxiliary}
T_N(2/3,4)\ge\left(\frac43\right)^5T_N(0,2).
\end{equation}
Indeed, multiplying~\eqref{eq:ch6-auxiliary} by $243T_N(2,0)$ and using~\eqref{eq:ch6-rank5-mw} gives
\[
T_M(2,0)T_M(0,2)\ge T_M(1,1)^2.
\]

Write $$T_N(x,y)=\sum_{i,j}t_{ij}x^iy^j.$$
We use coefficient nonnegativity, the generalized Brylawski identities

\begin{equation}\label{eq}
\sum_{i=0}^{h}\sum_{j=0}^{h-i}
\binom{h-i}{j}(-1)^jt_{ij}=0,
\qquad 0\le h<|E(N)|,
\end{equation}
see for example~\cite{BCCP24perm}, and the boundary coefficients for a simple connected rank-five matroid,$$
t_{5,0}=1,\qquad t_{4,0}=|E(N)|-5,\qquad
t_{0,|E(N)|-5}=1.$$
Since the difference between the two sides of~\eqref{eq} is linear in the coefficients $t_{ij}$, we formulate the required inequality as a linear program subject to these constraints. For $|E(N)|=10$ and $11$, exact rational Farkas certificates, together with the bound $T_N(0,2)\le 2^{|E(N)|}$ from the subset expansion, give the following lower bounds:
\begin{center}
\begin{tabular}{c|c}
$|E(N)|$ & certified lower bound for
$\displaystyle\frac{T_N(2/3,4)}{(4/3)^5T_N(0,2)}$\\ \hline
$10$ & $335023/262144$\\
$11$ & $1693215/1048576$
\end{tabular}
\end{center}
Both bounds are greater than $1$, proving~\eqref{eq:ch6-auxiliary}. The exact dual multipliers and a verifier using rational arithmetic are achieved with the computational material in \cite{Merino2026CyclicWidth}. This settles the last parameter tuples in~\eqref{eq:ch6-residual-parameters} and proves the theorem.
\end{proof}

The computations used in the cyclic-height $4$, $5$ and $6$ were performed in Python 3 using only the standard library. The source code and reproducibility scripts are archived in~\cite{MartinezSandoval2026CyclicHeight}.

\subsection{Matroids of small deficiency}\label{sec:deficiency}

For a nonempty loop- and coloop-free matroid $M$, let $g(M)$ and $g^*(M)=g(M^*)$ denote its girth and cogirth. We call
\[
\delta_g(M)=r(M)-g(M),\qquad
\delta_g^*(M)=r^*(M)-g^*(M)
\]
the \emph{girth deficiency} and \emph{cogirth deficiency}. Following Rajpal~\cite{Raj98}, a rank-$r$ matroid is $k$-paving if all its circuits have cardinality greater than $r-k$. Thus paving matroids are the $1$-paving matroids, while $4$-paving matroids are exactly those with $g(M)\ge r(M)-3$. See also~\cite{Sin25} for recent work on the $k$-paving hierarchy.

\begin{lem}\label{thm:large-girth}
If $M$ is loop- and coloop-free, then
\[
M\in\mathrm{CH}\bigl(3+\min\{\delta_g(M),\delta_g^*(M)\}\bigr).
\]
\end{lem}
\begin{proof}
We show that
\[
h(\mathcal Z(M))\leq3+\min\{\delta_g(M),\delta_g^*(M)\}.
\]
If $\mathcal Z(M)=\{\emptyset,E(M)\}$, the assertion is immediate. Otherwise take a longest chain
\[
\emptyset=Z_0\subsetneq Z_1\subsetneq\cdots\subsetneq Z_{h-2}\subsetneq Z_{h-1}=E(M)
\]
in $\mathcal Z(M)$, where $h=h(\mathcal Z(M))$. Since $Z_1$ is a nonempty cyclic flat, it contains a circuit and hence $r(Z_1)\geq g(M)-1$. Since $Z_{h-2}$ is a proper flat, $r(Z_{h-2})\leq r(M)-1$. Rank strictly increases along strict inclusions of flats, so $g(M)-1+(h-3)\leq r(M)-1$. Thus $h\leq r(M)-g(M)+3=\delta_g(M)+3$.

By~\cref{obs:dual}, complementation gives an order anti-isomorphism between $\mathcal Z(M)$ and $\mathcal Z(M^*)$, so they have the same height. Applying the preceding argument to $M^*$ gives $h\leq\delta_g^*(M)+3$.
\end{proof}

Combining~\cref{thm:large-girth,thm:cyclicheightsix} gives the following.


\fourpaving*


\section{Discussion and open problems}\label{sec:discussion}

\subsubsection*{Lattice-theoretic extensions}
It is natural to ask for an analogue of~\cref{thm:bounded-poset-quotient} with induced subposets replaced by sublattices. Such a statement cannot hold for every fixed lattice $L$: the cyclic-flat lattices of the counterexamples $U^{(2)}_{2n/3,n}$ have every proper principal ideal Boolean, and hence exclude $M_3^+$ and $N_5^+$ as sublattices, where $M_3^+$ and $N_5^+$ are obtained by adjoining a new maximum to $M_3$ and $N_5$, respectively, while their Merino--Welsh quotients are unbounded. On the other hand, for sufficiently large $n$ these lattices do contain $M_3$ and $N_5$, suggesting the following question.

\begin{quest}
Is $\Phi(M)$ uniformly bounded over loop- and coloop-free matroids $M$ for which $\mathcal Z(M)$ is distributive?
\end{quest}

One obstacle is that the property of excluding a fixed sublattice from the lattice of cyclic flats is not in general minor-closed, so \cref{lem:C(X)} does not generalize to sublattices. Indeed, let \(B_t\) be the \emph{\(t\)-page book graph}, consisting of \(t\) triangles sharing a common edge \(e\), and let \(M_t=M(B_t)\). Its cyclic-flat lattice is Boolean: \(\mathcal Z(M_t)\cong\mathcal B_t\). On deleting \(e\), however,

\[
\mathcal Z(M_t\setminus e)\cong
\{\varnothing\}\cup\{I\subseteq[t]:|I|\ge2\},
\]

ordered by inclusion. For \(t=3\) this lattice is \(M_3\). For \(t=4\), the five elements

\[
\varnothing,\qquad \{1,2\},\qquad \{1,2,3\},\qquad
\{3,4\},\qquad [4]
\]

form a sublattice isomorphic to \(N_5\). Thus neither the class of matroids with \(M_3\)-free cyclic-flat lattice nor that with \(N_5\)-free cyclic-flat lattice is minor-closed; in particular, the class whose cyclic-flat lattice is distributive is not minor-closed either.

\subsubsection*{Asymptotic behaviour}

The lower bounds in~\cref{cor:asymptotic-bounds} use only the examples of~\cite{BCCP24false}. It remains to determine the precise growth rates. Two simple refinements improve the lower-bound constants.

There is a natural self-dual modification of the counterexamples which improves the lower bound for cyclic width. Let $H_m$ be obtained from $U_{m,2m}$, whose ground set is partitioned into two $m$-sets $A$ and $B$, by replacing the elements of $A$ by parallel pairs and those of $B$ by coparallel pairs.

\begin{prop}\label{prop:mixed-uniform}
Setting $\eta=2\log_2(9/8)$, we have
\[
w(\mathcal Z(H_m))=\binom{m}{\lfloor m/2\rfloor}
\qquad\text{and}\qquad
\Phi(H_m)=2^{(\eta+o(1))m}.
\]
Consequently,

\[
\Phi(\mathrm{CW}(k))\geq k^{\eta-o(1)}.
\]

\end{prop}

\begin{proof}
The cyclic-flat lattice of $H_m$ consists of two copies of
$\mathcal B_m$, with the top of the first identified with the bottom of
the second. Hence

\[
w(\mathcal Z(H_m))
=\binom{m}{\lfloor m/2\rfloor}
=2^{m-o(m)}.
\]

If a basis of $U_{m,2m}$ contains $i$ elements of $A$, it gives rise to
$4^i$ bases of $H_m$. Thus

\[
T_{H_m}(1,1)
=\sum_{i=0}^m\binom mi^2 4^i
=2^{(\log_2 9+o(1))m},
\]

by Stirling's formula. Moreover, series extension on the elements of
$B$ gives

\[
\begin{aligned}
T_{H_m}(2,0)
&=\sum_{j=0}^m\binom mj2^j T_{U_{m,2m-j}}(2,0)\\
&=2\sum_{j=0}^m\binom mj2^j
  \sum_{i=0}^{m-1}\binom{2m-j-1}{i}
 =2^{(3+o(1))m}.
\end{aligned}
\]

Indeed, the upper bound follows by replacing the inner sum by
$2^{2m-j-1}$, while for $j=\lfloor m/2\rfloor$ the inner sum is
$2^{2m-j-1}(1-o(1))$ by binomial concentration, giving the matching lower bound. Since $H_m$ is self-dual, the
same estimate holds for $T_{H_m}(0,2)$. Therefore, $\Phi(H_m)
=2^{(2\log_2 9-6+o(1))m}
=2^{(\eta+o(1))m}$. 
Together with $w(\mathcal Z(H_m))=2^{m-o(m)}$, this yields the stated bound along the sequence $k=w(\mathcal Z(H_m))$. Since consecutive terms of this sequence differ by only a bounded factor and $\mathrm{CW}(k)$ is increasing in $k$, the same asymptotic bound holds for all $k$.

\end{proof}

The lower bounds for $\Phi(p)$ and $\Phi(\mathrm{CH}(k))$ can likewise
be sharpened by optimizing the rank in the original counterexample family.

\begin{prop}\label{prop:optimized-uniform}
Setting $\delta=\log_2\left(\frac{36+16\sqrt2}{49}\right)$, we have
\[
\Phi(p)\geq2^{(\delta-o(1))p}
\qquad\text{and}\qquad
\Phi(\mathrm{CH}(k))\geq2^{(\delta-o(1))k}.
\]
\end{prop}

\begin{proof}
For $1/2<\alpha<3/4$, let $M_n=U^{(2)}_{r,n}$, where $r=\lceil\alpha n\rceil$. By
\cite[Lemma~2.4]{BCCP24false}, at $x=2$,

\[
\Phi(M_n)=2^{(f(\alpha)+o(1))n},
\qquad
f(\alpha)=2\alpha+2H_2(\alpha)-3,
\]

where $H_2$ is the binary entropy function. Moreover,

\[
\mathcal Z(M_n)\cong
\{X\subseteq[n]:|X|\leq r-1\}\cup\{\hat1\},
\]

so its height is $r+1$ and it excludes $C_{r+2}$. Thus the exponential
rate with respect to either the forbidden-poset order or cyclic height
is $f(\alpha)/\alpha$.

Since

\[
\left(\frac{f(\alpha)}{\alpha}\right)'
=\frac{\log_2(8(1-\alpha)^2)}{\alpha^2},
\]

this ratio is maximized at
$\alpha=\alpha_*:=1-\frac1{2\sqrt2}$. At this value,

\[
\frac{f(\alpha_*)}{\alpha_*}
=\log_2\left(\frac{36+16\sqrt2}{49}\right)
=\delta.
\]

For the first assertion, for each $p$ take $r=p-2$ and choose an integer $n$ with $r/n\to\alpha_*$ as $p\to\infty$; then $M_n\in\mathcal C(C_p)$. For the second, take $r=k-1$ and choose $n$ with $r/n\to\alpha_*$ as $k\to\infty$; then $M_n\in\mathrm{CH}(k)$. This proves both assertions.
\end{proof}

Hence, with
\[
\eta=2\log_2(9/8)=0.339850\ldots,
\qquad
\delta=\log_2\left(\frac{36+16\sqrt2}{49}\right)=0.258793\ldots,
\]
we have
\[
2^{(\delta-o(1))p}\leq\Phi(p)\leq2^{4p-4},
\]

\[
2^{(\delta-o(1))k}\leq\Phi(\mathrm{CH}(k))\leq2^{4k-4},
\]

and

\[
k^{\eta-o(1)}
\leq\Phi(\mathrm{CW}(k))
\leq O\bigl(k^4(\log k)^2\bigr).
\]

\begin{quest}
What are the exact asymptotics in these three cases?
\end{quest}

\subsubsection*{Small values}
The exact formulas of Beke, Cs\'aji, Csikv\'ari, and Pituk also give the slightly smaller counterexample $U^{(2)}_{21,31}$ on $62$ elements, with $\Phi(U^{(2)}_{21,31})=1.0740\ldots>1$; compare~\cite{BCCP24false}, where $U^{(2)}_{22,33}$ is given as the smallest example known there. Its cyclic-flat lattice has height $22$. Consequently, our results and this example leave the following ranges open.

\begin{quest}
What is the smallest \(k\) for which the multiplicative Merino--Welsh inequality fails for some matroid in \(\mathrm{CH}(k)\)?
\end{quest}
The inequality holds for $k\leq6$ and fails already in $\mathrm{CH}(22)$. Thus the smallest such $k$ lies in the range $7\leq k\leq22$.

\begin{quest}
What is the smallest \(k\) for which the multiplicative Merino--Welsh inequality fails for some \(k\)-paving matroid? 
\end{quest}
The inequality holds for every $k$-paving matroid with $k\leq4$, whereas $U^{(2)}_{21,31}$ has rank $21$ and girth $2$, so it is $20$-paving and is a counterexample. Thus the smallest such $k$ lies in the range $5\leq k\leq20$.

\begin{quest}
What is the smallest \(k\) for which the multiplicative Merino--Welsh inequality fails for some matroid in \(\mathrm{CW}(k)\)?
\end{quest}

The inequality holds for $k\leq7$. Using the formulas in~\cref{prop:mixed-uniform}, exact evaluation gives $\Phi(H_{16})=1.0928\ldots>1$, while $w(\mathcal Z(H_{16}))=\binom{16}{8}=12870$. Thus the smallest such $k$ lies in the range $8\leq k\leq12870$.

\subsection*{Statement on the use of artificial intelligence}

Generative AI tools, primarily ChatGPT by OpenAI (including GPT-5.6 Sol), were used during the development of this work as interactive research assistants. The mathematical framework of the paper, including the choice of cyclic flats as the structural parameter, the induced-poset approach, the decorated-core decomposition and boundary matroids, as well as the main qualitative ideas and proof strategies, were developed by the authors previous to that.

AI assistance was used primarily to explore and push quantitative aspects of these ideas further. In particular, it contributed to the search for the strengthening of the initial reduction from classes of size at most three to the reduction to $2$-inflated matroids proved in~\cref{lem:inductionbasis}; to the optimization and sharpening of the quantitative and asymptotic bounds in~\Cref{cor:asymptotic-bounds,prop:mixed-uniform} and \Cref{prop:optimized-uniform}; and to the systematic exploration of the finite parameter ranges that extended the initial results for $\mathrm{CW}(4)$ and $\mathrm{CH}(4)$ to the present results~\cref{thm:cyclicwidthseven,thm:cyclicheightsix}. 

All statements, proofs, computations, and references appearing in the paper were subsequently checked and validated by the authors, who take full responsibility for the contents of the manuscript.

\section*{Acknowledgments}
The first author was supported by UNAM through the PREI-DGAPA. This work was supported by UNAM-PAPIIT IN119026 and in part by the grant PID2022-137283NB-C22 funded by MICIU/AEI/10.13039/\allowbreak 501100011033 and by ERDF/EU and through the Severo Ochoa and Mar\'ia de Maeztu Program for Centers and Units of Excellence in R\&D (CEX2020-001084-M).

\small
\bibliography{mwbib}
\bibliographystyle{siam}

\end{document}